\documentclass[11pt]{amsart}

\usepackage[margin=1in]{geometry}  
\usepackage{graphicx}              
\usepackage{amsmath}               
\usepackage{amsfonts}              
\usepackage{amsthm}      
\usepackage{mathtools}          
\usepackage{color}
\usepackage{tabulary}
\usepackage{caption}
\usepackage{subcaption}
\usepackage{amssymb}
\usepackage{todonotes}
\theoremstyle{plain}
\usepackage{mathptmx}
\usepackage{tikz-cd} 
\usepackage{mathrsfs}
\usepackage{verbatim}
\usepackage{tikz}
\usepackage{enumitem}
\usepackage{hyperref}
\usepackage{comment}

\newtheorem{thm}{Theorem}[section]
\newtheorem{lem}[thm]{Lemma}
\newtheorem{prop}[thm]{Proposition}
\newtheorem{cor}[thm]{Corollary}
\newtheorem{conj}[thm]{Conjecture}

\newtheorem{rmk}[thm]{Remark}

\newtheorem{ques}[thm]{Question}

\numberwithin{equation}{section}

\DeclareMathOperator{\Stab}{Stab}

\newcommand{\NN}{\mathbb{N}}      %
\newcommand{\ZZ}{\mathbb{Z}}      
\newcommand{\QQ}{\mathbb{Q}} 
\newcommand{\RR}{\mathbb{R}}      
\newcommand{\CC}{\mathbb{C}}      
\newcommand{\HH}{\mathbb{H}}      

\DeclareMathOperator{\PSL}{PSL}

\DeclareMathOperator{\SL2}{SL_2}

\DeclareMathOperator{\rad}{rad}

\begin{document}

\title[On non-freeness of groups generated by two parabolic matrices with rational parameters]{On non-freeness of groups generated by two parabolic matrices with rational parameters: limit points and the orbit test}


\author{Wonyong Jang}
\address{Department of Mathematical Sciences, KAIST,
291 Daehak-ro, Yseong-gu, 
34141 Daejeon, South Korea}
\email{jangwy@kaist.ac.kr}

\author{Dongryung Yi}
\address{Department of Mathematical Sciences, KAIST,
291 Daehak-ro, Yseong-gu, 
34141 Daejeon, South Korea}
\email{j-invariant@kaist.ac.kr}


\maketitle


\begin{abstract}
 For $\alpha \in \RR$, let $$G_{\alpha}:= \left< \begin{bmatrix}
1 & 1 \\ 0 & 1 \end{bmatrix} , \begin{bmatrix}
1 & 0 \\ \alpha & 1 \end{bmatrix}
\right> < \SL2(\RR).$$
 K. Kim and the first author established the orbit test, which provides a sufficient condition for $G_{\alpha}$ not to be a free group of rank 2.
 In this article, we present two main applications of the orbit test.
 First, using the corresponding modulo homomorphism, we show that the converse of the orbit test does not hold. 
 In particular, we construct explicit counterexamples, all of which are rational. 
 As another application, we prove that
$$ \frac{2 n_3 + 2 n_5 - 1}{n_3 n_4 (2 n_5 - 1)} \quad \quad (n_3, n_4 \neq 0)$$
is a limit point of limit points of non-free rational numbers.
 Moreover, we prove that
 $$3 + \frac{3}{2 (9 n - 1)} \quad \text{and} \quad 3 + \frac{9 n + 5}{3 (2 n + 1) (9 n + 4)}$$
 are non-free rational numbers which converge to $3$.
 Their construction relies on the orbit test together with a modified Pell's equation.
\end{abstract}

\vspace{1.2cm} 


\section{Introduction}
 For a complex number $\alpha$, consider the subgroup $G_{\alpha}$ of $\SL2 (\CC)$, generated by two parabolic matrices:
$$A := \begin{bmatrix} 1 & 1 \\ 0 & 1 \end{bmatrix}, \ B_{\alpha} := \begin{bmatrix} 1 & 0 \\ \alpha & 1 \end{bmatrix}, \ G_{\alpha} := \left< A, B_{\alpha} \right> < \SL2(\CC).$$
The problem of determining whether $G_{\alpha}$ is a free group of rank 2 or not has a long history and has been actively studied in the literature.
We say that $\alpha \in \CC$ is a \emph{free number} if $G_{\alpha}$ is a free group of rank 2, and that $\alpha \in \CC$ is a \emph{non-free number}, or a \emph{relation number}, otherwise.

 Many people have contributed to this problem, and different conventions have been used in the literature.
 For $\alpha_1, \alpha_2 \in \CC$, if we let
 $$A_{\alpha_1} := \begin{bmatrix} 1 & \alpha_1 \\ 0 & 1 \end{bmatrix}, \ B_{\alpha_2} := \begin{bmatrix} 1 & 0 \\ \alpha_2 & 1 \end{bmatrix}, \ H_{\alpha_1, \alpha_2} := \left< A_{\alpha_1}, B_{\alpha_2} \right> < \SL2(\CC),$$
 then $H_{\alpha_1, \alpha_2} \cong H_{1, \alpha_1 \alpha_2} = G_{\alpha_1 \alpha_2}$ \cite{chang1958certain}.
 Many previous results were obtained in the settings $H_{2, \lambda}, H_{\mu, \mu}$ and $H_{\alpha, 1}$.
Nonetheless, we adopt the notation $G_{\alpha} = H_{1, \alpha}$ since this formulation is most convenient for the orbit test, which is the main focus of this article.
Moreover, the above isomorphism tells us that all previous results can be converted into our setting.

The ultimate goal of this line of research is to characterize non-free numbers.
First, almost all complex numbers are free numbers due to the following results.

\begin{prop} \label{Large_Free}
 The following complex numbers are free:
\begin{enumerate}
    \item transcendental numbers \cite{fuchs1940certain},
    \item real numbers $\alpha$ with $|\alpha| \geq 4$ (\cite{sanov1947property} and \cite{brenner1955quelques}), and
    \item elements of the Riley slice of the Schottky space \cite{keen1994riley}.
\end{enumerate}
\end{prop}
\noindent We refer to \cite{lyubich1988free}, \cite{keen1994riley}, \cite{komori1998riley}, and \cite{akiyoshi2007punctured} for details on the Schottky space and the Riley slice.
 These results provide a geometric perspective on the problem.
 Also, readers can find simple descriptions of free numbers in \cite{chang1958certain}, \cite{lyndon1969groups}, and \cite{S1}.
 Note that if $a,b \in \CC$ are algebraically conjugate, then the corresponding Galois conjugation induces a group isomorphism $G_a \cong G_b$.
 Combining with Proposition \ref{Large_Free}, this implies that the set of free algebraic numbers is dense in $\CC$.
 On the other hand, it is known that the set of non-free numbers is dense in the open unit disk $\{ z \in \CC : |z|<1 \}$ by Ree \cite{ree1961certain}.
 Furthermore, the set of all non-free numbers is dense in the closed interval $[-4,4] \subset \RR$ \cite{chang1958certain}.
 However, it remains challenging to determine whether a given number is free or not.
 In particular, the following conjecture is a long-standing open problem.

\begin{conj} \label{M-Conj}
 Every rational number $\alpha$ with $|\alpha|<4$ is non-free.
\end{conj}

This conjecture was first suggested in \cite{lyndon1969groups} in the setting of the groups $H_{2,\lambda}$ and $H_{\mu,\mu}$.
Subsequently, the conjecture was reformulated by S. Kim and T. Koberda \cite{kim2022non}.
There has been considerable effort devoted to resolving this conjecture.
It is straightforward to show that the integers $1,2,$ and $3$ are non-free.
Moreover, Farbman proved in \cite{farbman1995non} that all rational numbers $\frac{q}{p}$ ($p, q \in \ZZ$) with $|q|\leq 16$ and $\left| \frac{q}{p} \right|<4$ are non-free.
This result is extended as follows:

\begin{thm}[Theorem 1.4 in \cite{kim2022non}] 
 Every rational number $\frac{q}{p}$ with $|q| \leq 27$, $|q| \neq 24$, and $\left| \frac{q}{p} \right|<4$ is non-free.
\end{thm}

In addition, there are many results on the characterization of non-free numbers via polynomials (\cite{riley1972parabolic}, \cite{brenner1975non}, \cite{bamberg2000non}). 
In particular, an interesting connection between non-free numbers and the generalized Chebyshev polynomials was established in \cite{choi2025non}.
This relationship was then used to prove that every rational number $\frac{q}{p}$ with $|p| \leq 22$ and $\left| \frac{q}{p} \right| <4$ is non-free with only six exceptions.
For further results on non-free rational numbers, we refer the reader to \cite[Theorem~7.7]{gilman2008structure}.

In \cite{MR4237486}, Kyeongro Kim and the first author introduced \emph{the orbit test}.
The orbit test provides a sufficient condition for a real number to be non-free.

\begin{prop}[The orbit test, Proposition 5.4 in \cite{MR4237486}]
Let $\alpha \in \RR$.
 If there exists $g \in G_{\alpha}$ such that $$g \cdot 0 = \frac{1}{2} \textnormal{ or } g \cdot \infty = \frac{1}{2},$$
 then $\alpha$ is non-free.
\end{prop}
For the detailed account, such as the action of $G_{\alpha}$ on $S^1 \cong \RR \cup \{ \infty \}$, we refer to Subsection~\ref{Pre_Orbit-Test} and \cite{MR4237486}.
Using this tool, many rational numbers, e.g., $$ \frac{41}{18}, \frac{57}{25}, \frac{59}{26}, \frac{43}{27},$$ are known to be non-free. See Table 1 in \cite{MR4237486}.
This naturally raises the question of whether the converse of the orbit test holds.
The first main result of this paper shows that the converse of the orbit test is false.
Moreover, the counterexamples turn out to be rational.

\begin{thm}[Corollary \ref{Converse_OT2}]
 The converse of the orbit test is false.
 Moreover, there are infinitely many non-free rational numbers that never satisfy the orbit test.
\end{thm}

A key observation from the previous result is that the numerator should be $\pm 2$ or odd (Proposition \ref{Converse_OT1}).
Consequently, although such counterexamples exist, the orbit test remains useful to show that $\alpha$ is non-free rational when the numerator is a large odd integer.
We point out that when $\alpha$ approaches to $4$, the length of all relations in $G_{\alpha}$ diverges to $\infty$ \cite{gutan2014diophantine}.
In Remark \ref{Compare_syllable_length}, we discuss the syllable length of a relation arising from the orbit test.
Thus, for rational numbers with a sufficiently large numerator, the orbit test becomes more effective. 
We also refer to \cite{yaari2025long} for further discussion of relations in $G_{\alpha}$.

In \cite{beardon1993pell}, A. F. Beardon revealed an interesting relationship between Pell's equation and non-free rational numbers.
This result was later extended in \cite{tan1996quadratic}. Both papers provide classes of convergent sequences of non-free rational numbers.
Smilga constructed sequences of non-free rational numbers converging to several real numbers in \cite{smilga2021new}. The list of limit points includes $2+\sqrt{2}$, which is currently the largest known limit point of non-free rational numbers.
Generalizing Smilga's work, further families of sequences of non-free rational numbers were obtained in \cite{buyalos2025family}.
Despite these contributions, it remains unknown whether there exists a sequence of non-free rational numbers converging to $3$ or $4$ (see \cite{choi2025non}).
In this paper, we address the case of $3$.

\begin{thm}[Corollary \ref{LimPtOfRRN3}] \label{LimPtOfRRN3-intro}
$3$ is a limit point of non-free rational numbers.
\end{thm}

\noindent Indeed, we prove that any rational number of the form $$ -\frac{2 n_3 + 2 n_5 - 1}{n_3 n_4 (2 n_5 - 1)} \quad (n_3, n_4 \in \ZZ - \{ 0 \}, \ n_5 \in \ZZ),$$
turns out to be a limit point of limit points of non-free rational numbers.
This result reveals that many rational numbers arise as limit points of limit points of non-free rational numbers.
Moreover, we find that the rational numbers
$$ 3 + \frac{3}{2 (9 n - 1)} \quad \textnormal{and} \quad
 3 + \frac{9 n + 5}{3 (2 n + 1) (9 n + 4)}$$
are non-free for all $n \in \ZZ$.
See Theorem \ref{LimPtOfRRN2} and Theorem \ref{ExplicitSeqsOfRRNConvTo3}, respectively.

After the initial submission of this manuscript to the arXiv,
we became aware that an equivalent result to Theorem~\ref{LimPtOfRRN3-intro} was independently obtained by Ilia Smilga in \cite{buyalos2025family}.
While the statements are equivalent, the proofs rely on different techniques: we use the orbit test, whereas their approach is based on half-relation polynomials.
In particular, our sequences are different from theirs.


The paper is organized as follows. In Section~\ref{Sec-Pre}, we briefly recall the orbit test, the corresponding principal congruence subgroup, and Pell's equation.
We use the orbit test and the corresponding modulo homomorphism to prove the first main result, Corollary \ref{Converse_OT2}, in Section~\ref{Sec-Cong&pi}.
Using a similar technique, we also characterize which $G_\alpha$ contains special types of matrices in this section.
In Section~\ref{Sec-Seq to 3}, we provide rational numbers that are limit points of non-free rational numbers (Theorem \ref{LimPtOfRRN2}).
As a corollary, $3$ turns out to be a limit point of non-free rational numbers  (Corollary \ref{LimPtOfRRN3}). 
In Section~\ref{Sec-ExplicitSeqTo3}, we explicitly construct sequences of non-free rational numbers that converge to $3$.

\vspace{0.6cm} 

\noindent \textbf{Acknowledgement} 
 Both authors are grateful to Hyungryul Baik and Sang-hyun Kim for helpful discussions and suggestions for qualitative improvement.
 We also thank Dongha Lee for helpful comments on an earlier draft of this manuscript.
 The first author has been supported by the National Research Foundation of Korea (NRF) grant funded by the Korea government (MSIT) RS-2025-00513595.
 The second author has been supported by $2025$ Summer-Fall KAIST Undergraduate Research Participation Program.


\section{Preliminaries} \label{Sec-Pre}
\subsection{The orbit test} \label{Pre_Orbit-Test}
In \cite{MR4237486}, the orbit test is established.
In this section, we briefly recall the orbit test and appeal to its usefulness.
Throughout this paper, we assume that $\alpha$ is real; hence $G_{\alpha}$ is a subgroup of $\SL2 (\RR)$.
The group $\SL2 (\RR)$ acts on the hyperbolic plane $\HH^2$, and this action induces an action on the circle at infinity $\partial \HH^2 \cong S^1 \cong \RR \cup \{ \infty \}$.
Consequently, $G_{\alpha}$ acts on $\RR \cup \{ \infty \}$, and a precise description of this action is given by the following formula:

 $$ \begin{bmatrix} a & b \\ c & d \end{bmatrix} \cdot r := \frac{ar+b}{cr+d}. $$

\noindent The following proposition is called the orbit test.
 
\begin{prop} [Proposition 5.4, \cite{MR4237486}] \label{Orbit_Test}
Let $\alpha \in \RR$.
 If there exists $g \in G_{\alpha}$ such that $$g \cdot 0 = \frac{1}{2} \textnormal{ or } g \cdot \infty = \frac{1}{2},$$
 then $\alpha$ is non-free.
\end{prop}

\begin{rmk} \label{Compare_syllable_length}
Several tests for non-freeness have been established. One commonly used technique is to find a matrix of the form $\begin{bmatrix} * & 0 \\ * & *
\end{bmatrix}$, 
other than $B_{\alpha}^n$,
or equivalently, in terms of the above action, an element $g \in G_\alpha-\left< B_{\alpha} \right>$ such that $g \cdot 0=0$.
In this case, we recover a relation of $G_\alpha$ from $g \cdot 0 = 0$ as follows.
Suppose $g=\begin{bmatrix} * & 0 \\ * & *
\end{bmatrix}$, then since $g \not \in \left< B_{\alpha} \right>$, we get $[gB_{\alpha}g^{-1},B_{\alpha}]=1.$
Note that the syllable length of the relation $[gB_{\alpha}g^{-1},B_{\alpha}]$ is approximately four times the syllable length of $g$.

 Now consider the orbit test so choose $g \in G_{\alpha}$ with $g \cdot 0 = \frac{1}{2}$. From this, we can recover an element of the form $\begin{bmatrix} * & 0 \\ * & *
\end{bmatrix}$ as follows.
Say $g=A^{n_1}B_{\alpha}^{n_2} \cdots B_{\alpha}^{n_{2k}}A^{n_{2k+1}}$. Then one can compute that
$$ L:= 
A^{n_{2k+1}} B_{\alpha}^{n_{2k}} \cdots B_{\alpha}^{n_2} A^{n_1} \
A^{-1} \
A^{n_1}B_{\alpha}^{n_2} \cdots B_{\alpha}^{n_{2k}}A^{n_{2k+1}}
=\begin{bmatrix} * & 0 \\ * & *
\end{bmatrix},$$
hence we obtain a relation $[LB_{\alpha}L^{-1},B_{\alpha}]$.
We point out that the syllable length of the relation is approximately eight times the syllable length of $g$.
Therefore, the orbit test becomes more useful to find relations with large syllable length.
As mentioned in the Introduction, this test becomes more effective for rational numbers with a sufficiently large numerator.
\end{rmk}

\subsection{The corresponding modulo homomorphisms}
 Now we assume that $\alpha$ is rational so put $\alpha = \frac{q}{p}$ where $p$ and $q$ are integers with $\gcd(p,q)=1$. 
 Recall that for fixed prime $r$, $G_{\frac{1}{r}} = \SL2 \left( \ZZ \left[ \frac{1}{r} \right] \right)$ so $G_{\frac{q}{r}} < \SL2 \left ( \ZZ \left[ \frac{1}{r} \right] \right )$ \cite[Proposition 2.1]{nyberg2023congruence}.
 Lemma \ref{G1OverpEqualSL2Z1Overp} tells us that this result can be extended to any nonzero integer $p$.

In \cite[p.204]{mennicke1967ihara}, it is proved that for prime $r$, $\SL2 \left( \ZZ \left[ \frac{1}{r} \right] \right)$ is generated by
 $$\begin{bmatrix} 1 & 0 \\ 1 & 1
 \end{bmatrix},\begin{bmatrix} 0 & 1 \\ -1 & 0
 \end{bmatrix}, \textnormal{ and } \begin{bmatrix} r & 0 \\ 0 & \frac{1}{r}
 \end{bmatrix}.$$
 Motivated by this result, we first provide the following generating set of $\SL2 \left( \ZZ \left[ \frac{1}{p} \right] \right)$:
 
\begin{lem}
\label{GeneratingSetOfSL2Z1Overp}
For any nonzero $p \in \ZZ$, $\SL2 \left( \ZZ \left[ \frac{1}{p} \right] \right)$ is generated by
$$ B_1 := \begin{bmatrix} 1 & 0 \\ 1 & 1
 \end{bmatrix}, \ R:= \begin{bmatrix} 0 & 1 \\ -1 & 0
 \end{bmatrix}, \ U_1 := \begin{bmatrix} r_1 & 0 \\ 0 & \frac{1}{r_1}
 \end{bmatrix}, \cdots , \ U_k := \begin{bmatrix} r_k & 0 \\ 0 & \frac{1}{r_k}
 \end{bmatrix},$$
 where $p = \pm r_1^{e_1} \cdots r_k^{e_k}$, $r_1, \cdots, r_k$ are distinct primes, and the exponents $e_i$ are positive integers.
\end{lem}

\begin{proof}
Choose a matrix $$X = \begin{bmatrix} x_1 & x_2 \\ x_3 & x_4
\end{bmatrix}\in \SL2 \left( \ZZ \left[ \frac{1}{p} \right] \right).$$
Then for any $P \in \ZZ$, $$ U_i^{P} X = \begin{bmatrix} x_1 r_i ^{P} & x_2 r_i ^{P} \\ x_3  r_i ^{-P} & x_4 r_i ^{-P}
\end{bmatrix}, $$
so for some $i_1,\cdots,i_t$, the matrix $X_1:=U_{i_1}^{P_1} \cdots U_{i_t}^{P_t} X$ has integer entries on the first row. Furthermore, we can make the entries on the first row of $X_1$ to be coprime by multiplying more $U_i$ factors. Suppose that the greatest common divisor $d$ of two entries on the first row of $X_1$ is greater than 1. Then we can express $X_1$ as 
$$ X_1 = \begin{bmatrix} da & db \\ u_1 & u_2
\end{bmatrix},$$ where $d,a,b \in \ZZ$ with $\gcd(a,b)=1$ and $u_1,u_2 \in \ZZ\left[ \frac{1}{p} \right]$. Since $\det(X_1) = 1$, $\frac{1}{d} = au_2-bu_1 \in \ZZ\left[ \frac{1}{p} \right]$, i.e., $d$ divides some power of $p$.
Hence we may assume that the entries on the first row of $X_1$ are coprime, possibly after left-multiplying $U_{i_1}^{P_1} \cdots U_{i_t}^{P_t} X$ by appropriate negative powers of $U_i$.

Next we obtain a matrix in $\SL2 \left( \ZZ \left[ \frac{1}{p} \right] \right)$, say $X_2$, whose first row is $\begin{bmatrix} \pm 1 & 0 \end{bmatrix}$ by right-multiplying $X_1$ by an appropriate word of $B_1$ and $R$. Indeed, we can apply the Euclidean algorithm on the first row $\begin{bmatrix} a & b \end{bmatrix}$ of $X_1$ to find the appropriate word, since $\gcd(a, b) = 1$ and
$$\begin{bmatrix} a & b \end{bmatrix} B_1 = \begin{bmatrix} a + b & b \end{bmatrix}, \ \begin{bmatrix} a & b \end{bmatrix} R = \begin{bmatrix} -b & a \end{bmatrix}.$$
Since $X_2 \in \SL2 \left( \ZZ \left[ \frac{1}{p} \right] \right)$, we have $X_2 = \begin{bmatrix} \pm 1 & 0 \\ x & \pm 1 \end{bmatrix}$ for some $x \in  \ZZ \left[ \frac{1}{p} \right] $. Then for any $1 \leq i \leq k$ and $P \in \ZZ$,
$$U_i^{-P} X_2 U_i^{P} = \begin{bmatrix} \pm 1 & 0 \\ x r_i^{2P} & \pm 1
\end{bmatrix}.$$
Therefore we obtain a matrix in $\SL2 (\ZZ)$, say $X_3$ by conjugating $X_2$ by suitable powers of $U_i$. Since $\SL2 (\ZZ)$ is generated by $B_1$ and $R$, the original matrix $X$ can be expressed as a word of $B_1, R, U_1, \cdots, U_k$.
\end{proof}

\begin{lem}
\label{G1OverpEqualSL2Z1Overp}
 For any nonzero integer $p$, $G_{\frac{1}{p}} = \SL2 \left( \ZZ \left[ \frac{1}{p} \right] \right)$.
\end{lem}
\begin{proof}
 Clearly, $G_{\frac{1}{p}}$ is a subgroup of $\SL2 \left( \ZZ \left[ \frac{1}{p} \right] \right)$.
 We shall prove the reverse inclusion $\SL2 \left( \ZZ \left[ \frac{1}{p} \right] \right) \subset G_{\frac{1}{p}}$. By Lemma \ref{GeneratingSetOfSL2Z1Overp}, it suffices to show that $G_{\frac{1}{p}}$ contains all of the following matrices:
$$ \begin{bmatrix} 1 & 0 \\ 1 & 1
 \end{bmatrix},\begin{bmatrix} 0 & 1 \\ -1 & 0
 \end{bmatrix},\begin{bmatrix} r_1 & 0 \\ 0 & \frac{1}{r_1}
 \end{bmatrix}, \cdots , \textnormal{ and } \begin{bmatrix} r_k & 0 \\ 0 & \frac{1}{r_k}
 \end{bmatrix}.$$
 The first two matrices are obviously elements of $G_{\frac{1}{n}}$.
Note that $r_i$ is prime. From \cite[p.204]{mennicke1967ihara}, we obtain 
$$ \begin{bmatrix} r_i & 0 \\ 0 & \frac{1}{r_i}
 \end{bmatrix} \in G_{\frac{1}{r_i}} \subset G_{\frac{1}{p}},$$
 for any $1 \leq i \leq k$.
 This means $ \SL2 \left( \ZZ \left[ \frac{1}{p} \right] \right) < G_{\frac{1}{p}}$.
 Therefore, we get $G_{\frac{1}{p}} = \SL2 \left( \ZZ \left[ \frac{1}{p} \right] \right)$.
  \end{proof}

\begin{rmk}
\normalfont
 From Lemma \ref{GeneratingSetOfSL2Z1Overp}, we can characterize when $G_\frac{1}{p}$ and $G_\frac{1}{p'}$ are equal as sets. For any $p, p' \in \ZZ - \{ 0 \}$, we have
 $$G_\frac{1}{p} = G_\frac{1}{p'} \iff \rad(p) = \rad(p').$$
 The `if' direction comes from the equality $G_{\frac{1}{p}} = \SL2 \left( \ZZ \left[ \frac{1}{\rad(p)} \right] \right)$.
 Indeed, we obtain
 $$G_{\frac{1}{p}} = \SL2 \left( \ZZ \left[ \frac{1}{p} \right] \right) = \langle B_1, R, U_1, \cdots, U_k \rangle = \SL2 \left( \ZZ \left[ \frac{1}{\rad(p)} \right] \right)$$
 using the generating set of $\SL2 \left( \ZZ \left[ \frac{1}{p} \right] \right)$ provided in Lemma \ref{GeneratingSetOfSL2Z1Overp}. Now suppose $\rad(p) \neq \rad(p')$. Without loss of generality, let $r$ be a prime number such that $r | p$ and $r \nmid p'$. Then we have $\begin{bmatrix} r & 0 \\ 0 & \frac{1}{r}
 \end{bmatrix} \in \SL2 \left( \ZZ \left[ \frac{1}{p} \right] \right) - \SL2 \left( \ZZ \left[ \frac{1}{p'} \right] \right)$. Thus the `only if' direction is also proved.
\end{rmk}
  
We can then consider the following surjective group homomorphism for a coprime pair $(p,q)$:
 $$ \pi : \SL2 \left( \ZZ \left[ \frac{1}{p} \right] \right) \to \SL2 \left( \ZZ \left[ \frac{1}{p} \right] \ \middle/ \ q \ZZ \left[ \frac{1}{p} \right] \right) \cong \SL2 (\ZZ/q \ZZ).$$
 Let $K_{\frac{q}{p}}$ be the kernel of $\pi_{\frac{q}{p}}:=\pi|_{G_{\frac{q}{p}}}.$ Then $K_{\frac{q}{p}}$ is a finite index normal subgroup of $G_{\frac{q}{p}}$ (actually the index is $q$) and is called the \emph{corresponding principal congruence subgroup} of $G_{\frac{q}{p}}$.
 The homomorphism $\pi_{\frac{q}{p}}$ is called the \emph{corresponding modulo homomorphism} of $G_{\frac{q}{p}}$.

\subsection{Pell's equation}
 We dedicate this subsection to introducing Pell's equation and its variation for our purpose.
In elementary number theory, the following fact is well known.

\begin{lem}[Pell's equation] \label{Pell_eqn}
 Let $d$ be a non-square positive integer. Then the following Diophantine equation
 $$ u^2 - dv^2 = 1 $$
 has infinitely many integer solutions $(u, v) \in \ZZ^2$. In particular, the equation has an integer solution that satisfies $v \neq 0$.
\end{lem}

\noindent We refer to \cite{barbeau2003pell} and \cite{burton2010ebook} for readers who are interested in Pell's equation. We use Pell's equation to prove the following observation.

\begin{lem}
\label{PellLem0}
Let $d$ be a non-square positive integer and $c$ be a nonzero integer. Suppose that $(u, v) = (u_0, v_0) \in \ZZ^2$ is one of the solutions of the following Diophantine equation
$$u^2 - d v^2 = c.$$
Then this Diophantine equation has infinitely many integer solutions $(u, v) \in \ZZ^2$ that satisfy the congruence condition
$$u \equiv u_0, \ v \equiv v_0 \pmod{d}.$$
\end{lem}

\begin{proof}
Since $d$ is a non-square positive integer, we can find integers $u_1, v_1 \in \ZZ$ such that $u_1^2 - d v_1^2 = 1$ and $v_1 \neq 0$, by Pell's equation, Lemma \ref{Pell_eqn}. Now we consider the quotient ring $R := \ZZ[\sqrt{d}] / (d)$ and denote the natural quotient map by $q: \ZZ[\sqrt{d}] \to R$. Since $R$ is a finite commutative ring, the group $R^\times$ of units of $R$ is finite. Also, $q(u_1 + v_1 \sqrt{d})$ is a unit of $R$ since $q(u_1 + v_1 \sqrt{d}) q(u_1 - v_1 \sqrt{d}) = q(1)$. Thus $q(u_1 + v_1 \sqrt{d}) \in R^\times$ has a finite order, say $t \in \NN$. 
Now we complete the proof by showing that every pair of integers $(u, v) \in \ZZ^2$ determined by the formula
\begin{align} \label{Generating_InfSol}
u + v \sqrt{d} = (u_0 + v_0 \sqrt{d}) (u_1 + v_1 \sqrt{d})^{t k} \in \RR \quad (k \in \ZZ)
\end{align}
provide distinct integer solutions of the Diophantine equation $u^2 - d v^2 = c$ which satisfy the congruence condition.

First, note that $u_1 + v_1 \sqrt{d} \neq \pm 1$ since $v_1 \neq 0$. It is obvious that the generated pairs $(u, v)$ are all distinct since $u_0 + v_0 \sqrt{d} \neq 0$. 
Next, by applying the norm $\mathrm{N}_{\QQ(\sqrt{d})/\QQ}$ to both sides in Formula \eqref{Generating_InfSol}, we get
$$u^2 - d v^2 = (u_0^2 - d v_0^2) (u_1^2 - d v_1^2)^{t k} = c \cdot 1^{t k} = c.$$
Thus the Diophantine equation is satisfied for such $(u, v)$.
By applying $q$ to both sides in Formula \eqref{Generating_InfSol}, we have 
$$ q(u + v \sqrt{d}) = q(u_0 + v_0 \sqrt{d}) $$ since $q(u_1 + v_1 \sqrt{d})^t = q(1)$. This shows that $(u - u_0) + (v - v_0) \sqrt{d} \in (d)$, namely, $u \equiv u_0, \ v \equiv v_0 \pmod{d}$. Therefore, we constructed infinitely many integer solutions $(u, v) \in \ZZ^2$ that satisfy the Diophantine equation and the congruence condition.
\end{proof}

 We now modify Pell's equation for our purpose as follows:

\begin{lem} \label{PellLem1}
Let $a_0, a_1, a_2 \in \ZZ$ be integers such that $a_1$ is even and $a_2$ is either 0 or a non-square positive integer.
Then the Diophantine equation
$$v^2 = a_2 u^2 + a_1 u + a_0$$
for $u, v \in \ZZ$ has infinitely many solutions if the equation has at least one solution.
 In particular, we can find infinitely many $u \in \ZZ$ such that $a_2 u^2 + a_1 u + a_0$ is a square integer if there is at least one such $u$.
\end{lem}
\begin{proof}
First, assume $a_2 = 0$. Then it is straightforward to check that the Diophantine equation has infinitely many solutions. Thus, we may assume $a_2 \neq 0$.
Hence, $a_2$ is a non-square positive integer. Recall that $a_1$ is even, so we can rewrite the equation as
\begin{align} \label{PellinLem}
\left( a_2 u + \frac{1}{2} a_1 \right)^2 - a_2 v^2 = \frac{1}{4} \left( a_1^2 - 4 a_0 a_2 \right)
\end{align}
Since $a_2$ is non-square, the above equation can be considered as 
\begin{align} \label{PellinLemC}
U^2-dV^2=c 
\end{align}
by letting $U = a_2 u+\frac{1}{2}a_1 $, $V = v$, $d = a_2$, and $c = \frac{1}{4} \left( a_1^2 - 4 a_0 a_2 \right)$.
Recall $U,V,d,c \in \ZZ$.
By the assumption, we have at least one solution $(U_0,V_0)$ for this Diophantine equation.
Then Lemma \ref{PellLem0} implies that Equation \eqref{PellinLemC} has infinitely many solutions for $(U,V)$ with $U \equiv U_0, \ V \equiv V_0 \pmod{d}$.
Since $v = V$, we obtain infinitely many $v \in \ZZ$. 
For $u$, note that $U \equiv U_0 \pmod{a_2}$, so for infinitely many $U \in \ZZ$, we can find $u$ from $U = a_2 u+\frac{1}{2}a_1$.
Therefore, we have infinitely many $(u,v) \in \ZZ^2$ as desired.
We complete the proof.
\end{proof}


\section{The orbit test and the corresponding modulo homomorphism} \label{Sec-Cong&pi}
 We devote this section to giving several results obtained from the orbit test, Proposition \ref{Orbit_Test}, and the corresponding modulo homomorphism.
 Throughout this section, we assume that $p,q$ are integers with $p \neq 0$ and $\gcd(p,q)=1$.
 Then the map $$ \pi : \SL2 \left( \ZZ \left[ \frac{1}{p} \right] \right) \to \SL2 \left( \ZZ \left[ \frac{1}{p} \right] \ \middle/ \ q \ZZ \left[ \frac{1}{p} \right] \right) \cong \SL2 (\ZZ/q \ZZ),$$
 induces a group homomorphism $$ \pi_{\frac{q}{p}} : G_{\frac{q}{p}} \to \SL2 (\ZZ/q \ZZ), $$ and this map is called the corresponding modulo homomorphism.
 As a first application, we show that the converse of the orbit test (Proposition \ref{Orbit_Test}) is false.
 Using a similar technique, we characterize when $G_\alpha$ contains special types of matrices, such as $-I$ and hollow matrices.

\subsection{The converse of the orbit test}
 In this subsection, we will show that the converse of the orbit test is false.
We investigate the condition for $\alpha=\frac{q}{p}$ to satisfy the orbit test.

\begin{prop} \label{Converse_OT1}
 Suppose that a group $G_{\frac{q}{p}}$ satisfies the orbit test. Then $|q|$ is either $2$ or an odd number.
\end{prop}
\begin{proof}
 Without loss of generality, we may assume $p,q > 0$.
 Put $$g=\begin{bmatrix}
a & b \\ c& d \end{bmatrix} \in G_{\alpha}.$$ First suppose $g \cdot 0=\frac{1}{2}$. This means $\frac{b}{d}=\frac{1}{2}.$ By using the homomorphism $\pi_{\frac{q}{p}},$ we have
$$ \pi_{\frac{q}{p}}(g)=\begin{bmatrix}
1 & k \\ 0 & 1 \end{bmatrix}$$ for some $0 \leq k < q$. Thus for some $n,a_1,a_2 \in \ZZ$, we have $$
b=k+a_1 \frac{q}{p^n} , \quad d=1+a_2 \frac{q}{p^n}.$$
Since $\frac{b}{d}=\frac{1}{2}$, we obtain $$ 2k+2 a_1 \frac{q}{p^n} = 1+a_2 \frac{q}{p^m},$$
i.e., $q(a_2-2a_1) = (2k-1)p^n$. Hence $q|(2k-1)$ since $\gcd(p, q) = 1$. This shows that $q$ is odd.

 Now assume that $g \cdot \infty=\frac{1}{2}$. This means $\frac{a}{c}=\frac{1}{2}$. Since 
 $$ \pi_{\frac{q}{p}}(g)=\begin{bmatrix}
1 & k \\ 0 & 1 \end{bmatrix}$$ for some $0 \leq k < q$,
we have
$$a=1+a_3\frac{q}{p^n} , \quad c=0+a_4 \frac{q}{p^n}$$
 for some $n,a_3,a_4 \in \ZZ$.
Since $\frac{a}{c}=\frac{1}{2}$, we obtain $$ 2+2a_3\frac{q}{p^n}=a_4 \frac{q}{p^n},$$
i.e., $q(a_4-2a_3)=2 p^n$. Hence, $q|2$ since $\gcd(p, q) = 1$. This shows that $|q| \in \{ 1, 2 \}$.

 We conclude that $|q|$ should be either $2$ or an odd number.
\end{proof}

 Recall that there are infinitely many non-free rational numbers whose numerator is even and between $4$ and $27$ (see \cite{kim2022non}).
 Thus, the following corollary follows immediately from the previous proposition.

\begin{cor} \label{Converse_OT2}
 The converse of the orbit test is false.
 Moreover, there are infinitely many non-free rational numbers that never satisfy the orbit test.
\end{cor}

\subsection{Special types of matrices}
 In this subsection, we investigate an equivalent condition for $G_\alpha$ to contain certain types of matrices.
 We begin with the negative identity matrix $-I$.
 
\begin{lem} \label{3--I}
Let $\alpha=\frac{q}{p}$ be a rational number.
Then $$-I \in G_{\alpha} \textnormal{ if and only if }
|q|=1 \textnormal{ or } 2.$$
\end{lem}
\begin{proof}
 The ``if" part is easy. Recall $-I \in G_1, G_2$ since $\left (B_2A^{-1} \right)^2=-I$.
 Conversely, assume $-I \in G_{\alpha}$ for rational $\alpha=\frac{q}{p}$.
 Note that $$\pi_{\frac{q}{p}}(-I)=\begin{bmatrix}
1 & k \\ 0 & 1 \end{bmatrix}$$ for some $0 \leq k < q$.
From the $(1,1)$-entry, we have
$$ -1 = 1 + a_1 \frac{q}{p^n}$$
for some $a_1,n \in \ZZ$. Thus $q | a_1 q = 2 p^n$. Since $\gcd(p,q)=1$, we obtain $q|2$, i.e., $|q| \in \{ 1, 2 \}$.
\end{proof}

Recall that when $\alpha=\frac{q}{p}$, $G_{\alpha}$ has a torsion element if and only if $|q|=1,2,3$ \cite[Theorem 1]{farbman1995non}.
Thus, Lemma \ref{3--I} studies a particular torsion element $-I$.
We point out that some previous works focused on subgroups of $\PSL_2 (\CC)$ (\cite{S1}, \cite{aimi2020classification}, and \cite{akiyoshi2021classification}), and the lifting problem of subgroups of $\PSL_2 (\CC)$ to $\SL2 (\CC)$ has attracted considerable attention (\cite{gilmanlifting} and \cite{andrew2025lifting}).
The importance of the element $-I$ lies in the fact that it is the unique nontrivial central torsion element of $\SL2 (\CC)$.
Its occurrence governs whether a subgroup of $\PSL_2(\CC)$ admits a lift to $\SL2 (\CC)$, making it a fundamental obstruction in the lifting problem.

 Next, we consider the hollow matrix case.
 Recall that a matrix is hollow if its diagonal entries are all zero.
 Indeed, we establish the following.

\begin{lem} \label{3-Hol-Diagonal}
Let $\alpha=\frac{q}{p}$ be a rational number.
Then $G_{\alpha}$ contains a matrix of the form either $\begin{bmatrix}
d & x \\ y & 0
\end{bmatrix}$ or $\begin{bmatrix}
0 & x \\ y & d
\end{bmatrix}$ if and only if $|q|=1$.
In particular, $G_{\alpha}$ contains a hollow matrix if and only if $|q|=1$.
\end{lem}
\begin{proof}
 If $\alpha=\frac{1}{p}$, then $G_{\alpha}=\SL2 \left( \ZZ \left[ \frac{1}{p} \right] \right)$ so obviously $G_{\alpha}$ contains a hollow matrix
 $\begin{bmatrix}
0 & -p \\ \frac{1}{p} & 0
\end{bmatrix}.$
Conversely, suppose $g:=\begin{bmatrix}
d & x \\ y & 0
\end{bmatrix} \in G_{\alpha}$.
Then we get $$\pi_{\frac{q}{p}}(g)=\begin{bmatrix}
1 & k \\ 0 & 1 \end{bmatrix}$$ for some $0 \leq k < q$.
From the $(2,2)$-entry, we have
$$ 0 = 1 + a_1 \frac{q}{p^n}. $$
This means $q|p^n$. Since $\gcd(p,q)=1$, $|q|$ should be $1$. The case of $\begin{bmatrix}
0 & x \\ y & d
\end{bmatrix}$ can be treated similarly.
\end{proof}

One may ask whether an analogous statement holds for diagonal matrices.
Unfortunately, we cannot obtain any analogous result in the case of diagonal matrices.
Indeed, by explicit calculation, we show that $G_{\alpha}$ contains a non-identity diagonal matrix when $\alpha=\frac{1}{m}+\frac{1}{n}$ for some $m,n \in \ZZ$.

\begin{lem} \label{3-Diagonal}
Let $\alpha=\frac{q}{p}$ be a rational number.
Then $G_{\alpha}$ contains a nontrivial diagonal matrix if $|\alpha|=\frac{1}{m}+\frac{1}{n}$.
\end{lem}
\begin{proof}
 Let $\alpha=\frac{1}{m}+\frac{1}{n}$ with $\alpha>0$ so we may assume $m,n > 0$.
 We shall find nonzero integers $a_1$, $b_1$, $a_2$, $b_2$, and $a_3$ such that the matrix
$$ W := A^{a_1} B_{\alpha}^{b_1} A^{a_2} B_{\alpha}^{b_2} A^{a_3} $$ is non-identity and diagonal.
We have $$ W = A^{a_1} B_{\alpha}^{b_1} A^{a_2} B_{\alpha}^{b_2} A^{a_3} = \begin{bmatrix}
* & f_2(\alpha) \\ f_1(\alpha)  & f_3(\alpha)
\end{bmatrix} $$
where
\begin{align}
\begin{aligned}
f_1(\alpha) &= a_2b_1b_2 \alpha^2 + (b_1+b_2)\alpha, \\
f_2(\alpha) &= (a_1a_2a_3b_1b_2)\alpha^2 + (a_1a_2b_1+a_1a_3b_1+a_1a_3b_2 + a_2a_3b_2)\alpha+ a_1 + a_2 + a_3, \\
f_3(\alpha) &= (a_2a_3b_1b_2)\alpha^2 + (a_2b_1+a_3b_1+a_3b_2)\alpha + 1 .
\end{aligned}
\end{align}
Since $W$ is diagonal, we have $f_1(\alpha) = f_2(\alpha) = 0$. From the vanishing of $f_1(\alpha)$, we have
$$-\frac{1}{a_2}\left( \frac{1}{b_1}+\frac{1}{b_2} \right) = \alpha=\frac{1}{m}+\frac{1}{n}.$$
We choose $a_2=1$, $b_1=-m$, and $b_2=-n$.

Since $f_1(\alpha) = 0$,
\begin{align}
\begin{aligned}
f_2(\alpha) & = (a_1a_2a_3b_1b_2)\alpha^2 + (a_1a_2b_1+a_1a_3b_1+a_1a_3b_2 + a_2a_3b_2)\alpha+ a_1 + a_2 + a_3 \\
& = a_1a_3 f_1(\alpha) + a_2(a_1b_1 + a_3b_2)\alpha + a_1 + a_2 + a_3 \\
& = a_2(a_1b_1 + a_3b_2)\alpha + a_1 + a_2 + a_3.
\end{aligned}
\end{align}
From the vanishing of $f_2(\alpha)$, we get $a_2(a_1b_1 + a_3b_2)\alpha + a_1 + a_2 + a_3 = 0$.
Since $\alpha=\frac{1}{m}+\frac{1}{n}$ and $(a_2, b_1, b_2) = (1, -m, -n)$, we have
$$ (m a_1 +n a_3)\left( \frac{1}{m}+\frac{1}{n} \right) - a_1 - a_3 = 1. $$
This is equivalent to the linear Diophantine equation $m^2a_1 + n^2a_3 = mn$, which has infinitely many integer solutions $(a_1, a_3)$ since $\gcd(m^2, n^2) = \gcd(m, n)^2 | mn$. 
Thus there exist nonzero integers $a_1$ and $a_3$ that makes $W$ diagonal.


Since $f_1(\alpha) = 0$,
$$f_3(\alpha) = (a_2a_3b_1b_2)\alpha^2 + (a_2b_1+a_3b_1+a_3b_2)\alpha + 1$$
$$= a_3 f_1(\alpha) + a_2 b_1 \alpha + 1 = a_2 b_1 \alpha + 1.$$

Since $(a_2, b_1, b_2) = (1, -m, -n)$, we have $f_3(\alpha) = -m \alpha + 1$. This implies that $f_3(\alpha) \neq 1$, or $W \neq I_2$, since $\alpha > 0$. Thus $G_{\alpha}$ contains a non-identity diagonal matrix when $\alpha=\frac{1}{m}+\frac{1}{n}$.
\end{proof}

\begin{rmk}
 Lemma~\ref{3-Diagonal} gives us the condition for $G_\alpha$ to contain a diagonal matrix.
 Similar to the previous lemmata, we want to find a partial converse of Lemma~\ref{3-Diagonal}.
 However, for given $\alpha=\frac{q}{p}$, if $X=\begin{bmatrix}
p^{\phi(q)} & 0 \\ 0 & p^{-\phi(q)}
\end{bmatrix}$, then we have $\pi_{\frac{q}{p}}(X) = \begin{bmatrix}
1 & 0 \\ 0 & 1
\end{bmatrix}.$
Thus, we cannot obtain any information about $\alpha$.
\end{rmk}

 We point out that the analogous result for a diagonal matrix is particularly significant from the following perspectives.
 One motivation is the computation of stabilizer subgroups. Recall the action of $G_{\alpha}$ on $S^1 \cong \RR \cup \{ \infty \}$ (see Subsection \ref{Pre_Orbit-Test} and \cite{MR4237486}).
 Using the action $G_{\alpha}$ on $S^1$, K. Kim and the first author constructed a new graph called the generalized Farey graph $\Gamma_{\alpha}$.
 The group $G_{\alpha}$ acts naturally on $\Gamma_{\alpha}$, and it is natural to ask how well this action behaves.
 For more details on the graph $\Gamma_\alpha$, we refer to \cite{MR4237486}.
 The analogous result for a diagonal matrix plays a key role in computing edge stabilizer subgroups.
 By the construction of $\Gamma_\alpha$, all edges in the graph are of the form $g \cdot \ell$ where $g \in G_{\alpha}$ and $\ell$ is the geodesic between $0$ and $\infty$ in $\overline{\HH^2}=\HH^2 \cup \partial \HH^2$.
 Thus, it is enough to calculate the stabilizer subgroup $\Stab_{G_{\alpha}}(\ell)$.
 Since every diagonal matrix in $G_\alpha$ lies in $\Stab_{G_{\alpha}}(\ell)$, determining which diagonal matrices belong to $G_{\alpha}$ is essential for computing this stabilizer.
 One of the purposes of Lemma \ref{3--I} and Lemma \ref{3-Hol-Diagonal} is precisely to find this stabilizer subgroup.
 An understanding of this stabilizer subgroup may lead to new algebraic information about $G_{\alpha}$, such as decompositions, in a manner analogous to Bass–Serre theory.
 In terms of the action $G_\alpha \curvearrowright \Gamma_{\alpha}$, we can extend Lemma~\ref{3-Hol-Diagonal} as follows.

\begin{prop}
 Let $\alpha=\frac{q}{p} \in \QQ$. Then the following are equivalent.
 \begin{enumerate}
    \item $|q|=1$, i.e., $|\alpha|=\frac{1}{p}$,
    \item $G_{\alpha} = \SL2 \left( \ZZ \left[ \frac{1}{p} \right] \right)$ as a set,
    \item $G_{\alpha}$ contains a hollow matrix,
    \item with respect to the action $G_\alpha \curvearrowright \Gamma_{\alpha}$, there exists an inversion element in $\Stab_{G_{\alpha}}(\ell)$,
    \item with respect to the action $G_\alpha \curvearrowright \Gamma_{\alpha}$, there exists an element which maps $0$ to $\infty$,
    \item with respect to the action $G_\alpha \curvearrowright \Gamma_{\alpha}$, there exists an element which maps $\infty$ to $0$.
 \end{enumerate}
\end{prop}

Another motivation comes from a conjecture of Carl-Fredrik Nyberg Brodda.
He proposed the following conjecture in \cite{nyberg2023congruence}.

\begin{conj}[Conjecture 1 in \cite{nyberg2023congruence}] \label{Carl-Conj}
 Let $\frac{q}{p} \in \QQ \cap (-4,4)$ with $\gcd(p,q)=1$.
 Then $G_{\frac{q}{p}}$ is equal to the congruence subgroup $$ \pi^{-1}\left( \left< \begin{bmatrix}
 1 & 1 \\ 0 & 1
 \end{bmatrix} \right>\right)$$ of $G_{\frac{1}{p}} = \SL2 \left( \ZZ \left[ \frac{1}{p} \right] \right)$.
 In particular, $G_{\frac{q}{p}}$ is a finite index subgroup of $G_{\frac{1}{p}}$.
\end{conj}
\noindent This conjecture suggests another approach to solving Conjecture \ref{M-Conj}.
 Recall that $G_{\frac{1}{p}} \cong \SL2 \left( \ZZ \left[ \frac{1}{p} \right] \right)$.
 This group is never virtually free since $G_{\frac{1}{p}}$ contains $\ZZ \left[ \frac{1}{p} \right]$ as a subgroup.
 (Even, it is not hyperbolic.)
 This implies that if $G_{\frac{q}{p}}$ is a finite index subgroup of $G_{\frac{1}{p}}$, then $\alpha=\frac{q}{p}$ is non-free. Namely, Conjecture \ref{Carl-Conj} implies Conjecture \ref{M-Conj}.
 
 In order to disprove this conjecture, it suffices to find a coprime pair $(p,q)$ with $\left|\frac{q}{p}\right|<4$ such that the index $\left [ G_{\frac{1}{p}} : G_{\frac{q}{p}} \right ] $ is infinite.
 One approach to establishing this is to prove that $G_{\frac{q}{p}}$ has no non-identity diagonal matrix.
 Let $D$ be the set of all diagonal matrices in $G_{\frac{1}{p}} = \SL2 \left( \ZZ[\frac{1}{p}] \right)$. Then $D$ is an infinite subgroup of $G_{\frac{1}{p}}$. 
 If $G_{\frac{q}{p}}$ is a finite index subgroup of $G_{\frac{1}{p}}$, then the intersection $D \cap G_{\frac{q}{p}}$ is a finite index subgroup of $D$, hence $D \cap G_{\frac{q}{p}}$ should be infinite.
 Therefore, if we prove that $G_{\frac{q}{p}}$ has no non-identity diagonal matrix, then $\left [ G_{\frac{1}{p}} : G_{\frac{q}{p}} \right ] = \infty $.
 This observation suggests that our strategy may be useful in disproving Conjecture~\ref{Carl-Conj}.
 We refer to \cite{hao2026groups} for readers who are interested in a partial answer to Conjecture~\ref{Carl-Conj}.

 We finish the section by posing the following problem concerning diagonal matrices.

\begin{ques} \label{Pro-Diagonal}
Let $\alpha=\frac{q}{p}$ be a rational number.
Then find an equivalent condition of $\alpha$ such that $G_{\alpha}$ contains a non-identity diagonal matrix.
\end{ques}

If Conjecture~\ref{Carl-Conj} is true, then the answer to the above question would be $|\alpha|<4$.


\section{Sequences of non-free rational numbers} \label{Sec-Seq to 3}
 The main purpose of this section is to show that $3$ is a limit point of non-free rational numbers. 
 We suggest real numbers that are limit points of non-free rational numbers.

\begin{prop} \label{LimPtOfRRN1}
Let $y, n_3, n_4 \in \ZZ - \{ 0 \}$, $n_5 \in \ZZ$ be integers such that
$$(s y + n_4 r)^2 - 8 n_3 n_4 r y$$
is either 0 or a non-square positive integer, where $r := 2 n_5 - 1$ and $s := 2 n_3 + 2 n_5 - 1$.
Then any of two real numbers
$$\frac{-(s y + n_4 r) \pm \sqrt{(s y + n_4 r)^2 - 8 n_3 n_4 r y}}{2 n_3 n_4 r y}$$
is a limit point of non-free rational numbers.
\end{prop}
\begin{proof}
 By the orbit test (Proposition \ref{Orbit_Test}), every real number $\alpha$ satisfying
$$ A^{n_5} B_\alpha^{n_4} A^{n_3} B_\alpha^{n_2} A^{n_1} \cdot 0 = \frac{1}{2} $$
is non-free.
 We now adjust the parameters $n_1, n_2, n_3, n_4, n_5 \in \ZZ$ to make $\alpha$ rational.
We rewrite it as follows:
$$
\begin{bmatrix} 2 & -1 \end{bmatrix}
\begin{bmatrix} 1 & n_5 \\          0 & 1 \end{bmatrix}
\begin{bmatrix} 1 &   0 \\ n_4 \alpha & 1 \end{bmatrix}
\begin{bmatrix} 1 & n_3 \\          0 & 1 \end{bmatrix}
\begin{bmatrix} 1 &   0 \\ n_2 \alpha & 1 \end{bmatrix}
\begin{bmatrix} 1 & n_1 \\          0 & 1 \end{bmatrix}
\begin{bmatrix} 0 \\ 1 \end{bmatrix}
= 0.$$
Then we have
\begin{align}
\begin{aligned}
(2 n_1 n_2 n_3 n_4 n_5 - n_1 n_2 n_3 n_4) \alpha^2 \\
+ (2 n_1 n_2 n_3 + 2 n_1 n_2 n_5 + 2 n_1 n_4 n_5 + 2 n_3 n_4 n_5 - n_1 n_2 - n_1 n_4 - n_3 n_4)  \alpha \\
+ (2 n_1 + 2 n_3 + 2 n_5 - 1 )  = 0.
\end{aligned}
\end{align}

\noindent Now we consider $n_1$ and $n_2$ as variables, whereas $n_3$, $n_4$, and $n_5$ would be regarded as constants. To distinguish variables and constants, let $x := n_1 \in \ZZ$ and $y := n_2 \in \ZZ$.
For convenience, we put $r := 2 n_5 - 1 \in \ZZ$ and $s := 2 n_3 + 2 n_5 - 1 \in \ZZ$.
After the substitutions, we have
$$(n_3 n_4 r x y) \alpha^2 + (s x y + n_4 r x + n_3 n_4 r) \alpha + (2 x + s) = 0.$$

\noindent Now suppose $x$, $y$, $n_3$, and $n_4$ are nonzero. Then $n_3 n_4 r x y \neq 0$ so the above equation is quadratic for $\alpha$. Let $\alpha_{\pm}$ be two roots of this quadratic equation, namely,
\begin{align} \label{alpha_pm}
\alpha_\pm := \frac{-(s x y + n_4 r x + n_3 n_4 r) \pm \sqrt{(s x y + n_4 r x + n_3 n_4 r)^2 - 4 (n_3 n_4 r x y) (2 x + s)}}{2 n_3 n_4 r x y}.
\end{align}

\noindent Thus, when the value $$ (s x y + n_4 r x + n_3 n_4 r)^2 - 4 (n_3 n_4 r x y) (2 x + s) $$ is a square integer, $\alpha_{\pm}$ is rational.
 To find such square integers, consider the following Diophantine equation 
\begin{align} \label{LemInPell1}
        z^2 = (s x y + n_4 r x + n_3 n_4 r)^2 - 4 (n_3 n_4 r x y) (2 x + s),
\end{align}

\noindent where $z \in \ZZ$. By arranging the equation with respect to $x$, we get
$$z^2 = \left( (s y + n_4 r)^2 - 8 n_3 n_4 r y \right) x^2 + \left( 2 n_3 n_4 r (-s y + n_4 r) \right) x + (n_3 n_4 r)^2.$$ 

\noindent By the assumption, the value $(s y + n_4 r)^2 - 8 n_3 n_4 r y$ is either $0$ or a non-square positive integer.
 Notice that $2 n_3 n_4 r (-s y + n_4 r)$ is even, and $(x, z) = (0, \pm n_3 n_4 r)$ provide examples of solutions of Diophantine equation \eqref{LemInPell1}. 
 By Lemma \ref{PellLem1}, we conclude that there are infinitely many $x \in \ZZ$ satisfying Equation \eqref{LemInPell1}.
Therefore, we can take the limit $\lim_{|x| \to \infty} \alpha_\pm$, and we conclude that 
the real numbers $$\lim_{|x| \to \infty} \alpha_\pm = \frac{-(s y + n_4 r) \pm \sqrt{(s y + n_4 r)^2 - 8 n_3 n_4 r y}}{2 n_3 n_4 r y}$$ are indeed limit points of non-free rational numbers.
\end{proof}

\begin{rmk}
\normalfont
 From Proposition~\ref{LimPtOfRRN1}, we can recover Smilga's result in \cite{smilga2021new}.
 Namely, putting $n_3 = -1, n_4 = 1$, and $n_5 = 0$, we obtain
 $$\lim_{|x| \to \infty} \alpha_\pm = \frac{(3 y + 1) \pm \sqrt{(3y+1)^2 - 8 y}}{2 y}.$$
 Then, $y=1$ yields $2 \pm \sqrt{2}$ so $2+\sqrt{2}$ is a limit point of non-free rational numbers.
\end{rmk}

 The next lemma tells us that the condition of $(s y + n_4 r)^2 - 8 n_3 n_4 r y$ is not strict.

\begin{lem} \label{LimPtOfRRN1+}
Let $n_3, n_4 \in \ZZ - \{ 0 \}$, $n_5 \in \ZZ$ be integers, and put $r := 2 n_5 - 1$ and $s := 2 n_3 + 2 n_5 - 1.$ Then
$$(s y + n_4 r)^2 - 8 n_3 n_4 r y$$
is a non-square positive integer for all but finitely many $y \in \ZZ - \{ 0 \}$.
\end{lem}
\begin{proof}
 We regard $n_3$, $n_4$, and $n_5$ as constants.
 By arranging the equation with respect to $y$, we get
$$(s y + n_4 r)^2 - 8 n_3 n_4 r y = s^2 y^2 + 2 n_4 r (s - 4 n_3) y + (n_4 r)^2.$$
Since $s = 2 n_3 + 2 n_5 - 1$ is an odd integer, the coefficient $s^2$ of $y^2$ is positive.
Therefore, $s^2 y^2 + 2 n_4 r (s - 4 n_3) y + (n_4 r)^2$ is a positive integer for infinitely many $y \in \ZZ - \{ 0 \}$.

Now suppose that for given $y$, $s^2 y^2 + 2 n_4 r (s - 4 n_3) y + (n_4 r)^2$ is a square of an integer $w \in \ZZ$. 
That is, for $y \in \ZZ - \{ 0 \}$, there exists $w \in \ZZ$ satisfying
\begin{align} \label{LemInPell2}
        w^2 = s^2 y^2 + 2 n_4 r (s - 4 n_3) y + (n_4 r)^2.
\end{align}

Then we have
\begin{align}
\begin{aligned}
& w^2 = s^2 y^2 + 2 n_4 r (s - 4 n_3) y + (n_4 r)^2 \\
&\iff s^2 w^2 = s^4 y^2 + 2 n_4 r s^2 (s - 4 n_3) y + (n_4 r s)^2 \\
&\iff (s w)^2 = \left( s^2 y + n_4 r(s-4 n_3) \right)^2 - n_4^2 r^2(s-4n_3)^2 + (n_4 r s)^2  \\
&\iff \left( s w + s^2 y + n_4 r (s - 4 n_3) \right) \left( s w - s^2 y - n_4 r (s - 4 n_3) \right) = 8 n_3 (n_4 r)^2 (s - 2 n_3) \\
&\iff \left( s w + s^2 y + n_4 r (s - 4 n_3) \right) \left( s w - s^2 y - n_4 r (s - 4 n_3) \right) = 8 n_3 n_4^2 r^3.
\end{aligned}
\end{align}
The last equivalence follows from $s - 2 n_3 = 2 n_5 - 1 = r$.
Thus, whenever we find $w, y \in \ZZ$ satisfying \eqref{LemInPell2}, we confirm that $s w \pm (s^2 y + n_4 r (s - 4 n_3))$ are divisors of the nonzero integer $8 n_3 n_4^2 r^3$.
Since $8 n_3 n_4^2 r^3 = 8 n_3 n_4^2 (2 n_5 - 1)^3$ is fixed, there are only finitely many divisors of $8 n_3 n_4^2 r^3$.
However, in our factorization of $8 n_3 n_4^2 r^3$, the difference between two divisors is 
$$ 2 s^2 y + 2 n_4 r (s - 4 n_3),$$
which depends on $y$.
 Therefore, for given $n_3$, $n_4$, and $n_5$, there are only finitely many integer pairs $(y,w) \in \ZZ^2$ satisfying Equation \eqref{LemInPell2}.
\end{proof}

 Now we give rational numbers that are limit points of non-free rational numbers.

\begin{thm} \label{LimPtOfRRN2}
Let $n_3, n_4 \in \ZZ - \{ 0 \}$, $n_5 \in \ZZ$ be integers, and put $r := 2 n_5 - 1$ and $s := 2 n_3 + 2 n_5 - 1$.
Then a rational number
$$-\frac{s}{n_3 n_4 r} = -\frac{2 n_3 + 2 n_5 - 1}{n_3 n_4 (2 n_5 - 1)}$$
is a limit point of limit points of non-free rational numbers.
\end{thm}
\begin{proof}
Fix $n_3, n_4 \in \ZZ - \{ 0 \}$, $n_5 \in \ZZ$, and choose $y \in \ZZ$ be a nonzero integer satisfying $$(s y + n_4 r)^2 - 8 n_3 n_4 r y$$
is either 0 or a non-square positive integer using Lemma \ref{LimPtOfRRN1+}.
Then by Proposition \ref{LimPtOfRRN1}, the real number
$$\lim_{|x| \to \infty} \alpha_\pm = \frac{-(s y + n_4 r) \pm \sqrt{(s y + n_4 r)^2 - 8 n_3 n_4 r y}}{2 n_3 n_4 r y} $$
is a limit point of non-free rational numbers. Recall that $\alpha_\pm$ is given in Formula \eqref{alpha_pm}.

We again use Lemma \ref{LimPtOfRRN1+} to choose infinitely many $y_k \in \ZZ$ so that $(s y_k + n_4 r)^2 - 8 n_3 n_4 r y_k$ is either $0$ or a non-square positive integer.
Thus, we can take the limit of the values
\begin{align} 
(L_k)_\pm := \lim_{|x| \to \infty} \alpha_\pm = \frac{-(s y_k + n_4 r) \pm \sqrt{(s y_k + n_4 r)^2 - 8 n_3 n_4 r y_k}}{2 n_3 n_4 r y_k}. 
\end{align}
Letting $k \to \infty$, observe 
$$\lim_{|y| \to \infty} (L_k)_\pm = \frac{-s \pm \sqrt{s^2}}{2 n_3 n_4 r},$$ 
so one of the limit is $$ L := -\frac{s}{n_3 n_4 r}.$$
Since $(L_k)_{\pm}$ is a limit point of non-free rational numbers, we obtain that the number 
$$L = -\frac{s}{n_3 n_4 r} = -\frac{2 n_3 + 2 n_5 - 1}{n_3 n_4 (2 n_5 - 1)}$$
turns out to be a limit point of limit points of non-free rational numbers.
In other words, $L$ is a limit point of non-free rational numbers.
\end{proof}

 This theorem immediately shows the following.

\begin{cor} \label{LimPtOfRRN3}
3 is a limit point of non-free rational numbers.
\end{cor}
\begin{proof}
Take $n_3 = -1$, $n_4 = 1$, and $n_5 = 0$ in Theorem \ref{LimPtOfRRN2}.
\end{proof}

\noindent From our contribution, the following natural problem arises.

\begin{ques}
 Can we construct a sequence of non-free rational numbers that converges to $4$? 
\end{ques}

\noindent We emphasize that $4$ is a limit point of non-free algebraic integer numbers \cite[Theorem 6.1]{MR4237486}.
 But still, our understanding of non-free rational numbers remains fairly limited.
In particular, Conjecture \ref{M-Conj} is still open, and even the following conjecture is far from being well understood.

\begin{conj} \label{D-Conj}
 Find a dense subset $S \subset \QQ \cap (-4,4)$ consisting of non-free numbers.
\end{conj}

\noindent Conjecture \ref{D-Conj} would be useful to treat Conjecture \ref{M-Conj}, but still, we do not know the answer.
Furthermore, we do not know any examples of open intervals $(a,b) \subset (-4,4)$ such that $\alpha$ is non-free when $\alpha \in (a,b) \cap \QQ$.


\section{Explicit sequences of non-free rational numbers converging to $3$} \label{Sec-ExplicitSeqTo3}

In this section, we explicitly construct the sequence of non-free rational numbers converging to $3$. We start with the following lemma, which is a refinement of Lemma \ref{LimPtOfRRN1+} for the $(n_3, n_4, n_5) = (-1, 1, 0)$ case.
Recall $r := 2 n_5 - 1 = -1$ and $s := 2 n_3 + 2 n_5 - 1 = -3$.

\begin{lem} \label{EqnSolvingLem1}
For every nonzero $y \in \ZZ$,
$$(s y + n_4 r)^2 - 8 n_3 n_4 r y = 9 y^2 - 2 y + 1$$
is a non-square positive integer.
\end{lem}

\begin{proof}
We have $9 y^2 - 2 y + 1 > 0$ for any integer $y$. Now suppose that there is an integer $w \in \ZZ$ such that
$$w^2 = 9 y^2 - 2 y + 1$$
for some integer $y \in \ZZ$. Then we can multiply both sides by 9 and rewrite the equation as follows:
$$(3 w)^2 = (9 y - 1)^2 + 8$$
$$\iff (3 w + 9 y - 1) (3 w - 9 y + 1) = 8.$$
Since the sum of $3 w + 9 y - 1$ and $3 w - 9 y + 1$ is $6 w$, $3 w \pm (9 y - 1)$ would be integers whose product is 8 and sum is a multiple of 6. 
Thus, one of $3 w \pm (9 y - 1)$ is $\pm 2$, and the other one is $\pm 4$. 
Hence, their sum must be $\pm 6$, meaning that $w$ should be $\pm 1$. 
Therefore, $y \in \ZZ$ satisfies $9 y^2 - 2 y + 1 = w^2 = 1$, which implies $y = 0$. We conclude that for every $y \in \ZZ - \{ 0 \}$, $9 y^2 - 2 y + 1$ is a non-square positive integer.
\end{proof}

In Section~\ref{Sec-Seq to 3}, we proved that there exists a sequence of non-free rational numbers converging to 3. The following theorem provides two \textit{explicit} sequences of non-free rational numbers converging to 3.

\begin{thm} \label{ExplicitSeqsOfRRNConvTo3}
For any integer $n \in \ZZ$, the rational numbers
$$\frac{3 (18 n - 1)}{2 (9 n - 1)} = 3 + \frac{3}{2 (9 n - 1)},$$
$$\frac{162 n^2 + 162 n + 41}{3 (2 n + 1) (9 n + 4)} = 3 + \frac{9 n + 5}{3 (2 n + 1) (9 n + 4)}$$
are non-free. 
\end{thm}

\begin{proof}
We start with the same argument as in the proof of Proposition \ref{LimPtOfRRN1}. 
Let $n_1$, $n_2$, $n_3$, $n_4$, and $n_5$ be integers and observe that any real number $\alpha$ satisfying
$$A^{n_5} B_\alpha^{n_4} A^{n_3} B_\alpha^{n_2} A^{n_1} \cdot 0 = \frac{1}{2}$$
is non-free, by the orbit test (Proposition \ref{Orbit_Test}). 
Next, we regard $x := n_1 \in \ZZ$ and $y := n_2 \in \ZZ$ as variables and $n_3$, $n_4$, and $n_5$ as constants. 
For convenience, we again put $r := 2 n_5 - 1$ and $s := 2 n_3 + 2 n_5 - 1.$ 
If $x$, $y$, $n_3$, and $n_4$ are nonzero, then the real numbers
$$\alpha_\pm := \frac{-(s x y + n_4 r x + n_3 n_4 r) \pm \sqrt{(s x y + n_4 r x + n_3 n_4 r)^2 - 4 (n_3 n_4 r x y) (2 x + s)}}{2 n_3 n_4 r x y}$$
are non-free (c.f. Formula \eqref{alpha_pm}).
Again, the key point is to make the value
$$(s x y + n_4 r x + n_3 n_4 r)^2 - 4 (n_3 n_4 r x y) (2 x + s)$$
a square integer.

We fix the values $n_3 = -1$, $n_4 = 1$, and $n_5 = 0$.
Then we have $r = -1$, $s = -3$, and
$$\alpha_\pm = \frac{(3 x y + x - 1) \pm \sqrt{(3 x y + x - 1)^2 - 4 x y (2 x - 3)}}{2 x y}.$$
By arranging the value $(3 x y + x - 1)^2 - 4 x y (2 x - 3)$ with respect to $x$, we get
\begin{align} \label{NinSqRt}
(9 y^2 - 2 y + 1) x^2 + (6 y - 2) x + 1.
\end{align}
By Lemma \ref{EqnSolvingLem1}, the coefficient $9 y^2 - 2 y + 1$ of $x^2$ is a non-square positive integer for any nonzero $y \in \ZZ$.
Therefore, we can always find some $x \in \ZZ - \{ 0 \}$ that makes Formula \eqref{NinSqRt} a square integer for any $y \in \ZZ - \{ 0 \}$, as in the proof of Theorem \ref{LimPtOfRRN2}.
The following claim provides an explicit integer $x \in \ZZ - \{ 0 \}$ that makes Formula \eqref{NinSqRt} a square integer, for each $y \in \ZZ - \{ 0 \}$.

\vspace{0.6cm}
\textbf{Claim:} For any $y \in \ZZ - \{ 0 \}$, put \\
\begin{align}
\begin{aligned}
x &= \begin{cases}
    -\frac{ 27}{2} y   +       3     & \quad (y \equiv 0 \pmod{2}) \\
    -\frac{243}{4} y^2 + \frac{3}{4} & \quad (y \equiv 1 \pmod{2}) \\
\end{cases} \\
&= \begin{cases}
    -\frac{3}{2} (9 y - 2)           & \quad (y \equiv 0 \pmod{2}) \\
    -\frac{3}{4} (9y-1)(9y+1) & \quad (y \equiv 1 \pmod{2}) \\
\end{cases}.
\end{aligned}
\end{align}
\indent Then Formula \eqref{NinSqRt} is a square integer. \\
\textit{Proof of the Claim.}
 Assume that $y$ is even. Then we get $$(9 y^2 - 2 y + 1) x^2 + (6 y - 2) x + 1 = \frac{1}{4} (81 y^2 - 27 y + 4)^2,$$
 which is a square integer. Next, assume that $y$ is odd. Then we have
$$(9 y^2 - 2 y + 1) x^2 + (6 y - 2) x + 1 = \frac{1}{16} (729 y^3 - 81 y^2 + 27 y + 1)^2,$$
which is a square integer. This completes the proof. \hfill $\blacksquare$
\vspace{0.6cm}

To apply the claim, first suppose that $y$ is even. 
Since $81 y^2 - 27 y + 4 > 0$ for all $y \in \ZZ$, we obtain
\begin{align}
\begin{aligned}
    \alpha_-
    & = \frac{-2 (3 x y + x - 1) + \sqrt{4 ((3 x y + x - 1)^2 - 4 x y (2 x - 3))}}{-4 x y} \\
    & = \frac{(81 y^2 + 9 y - 4) + \sqrt{(81 y^2 - 27 y + 4)^2}}{6 y (9 y - 2)} \\
    & = \frac{162 y^2 - 18 y}{6 y (9 y - 2)} = \frac{3 (9 y - 1)}{9 y - 2} = 3 + \frac{3}{9 y - 2}.
\end{aligned}
\end{align}
Since $y$ is a nonzero even integer, putting $y = 2 n$, we conclude that every rational number of the form
$$\frac{3 (18 n - 1)}{2 (9 n - 1)} = 3 + \frac{3}{2 (9 n - 1)}$$
is non-free for any $n \in \ZZ - \{ 0 \}$. When $n = 0$, $3 + \frac{3}{2 (9 \cdot 0 - 1)} = \frac{3}{2}$. It is already known to be non-free so we completes the proof for the first family.

Next, we suppose that $y$ is odd.
Then we have
\begin{align}
\begin{aligned}
     \alpha_\pm 
    & =  \frac{-4 (3 x y + x - 1) \mp \sqrt{16 ((3 x y + x - 1)^2 - 4 x y (2 x - 3))}}{-8 x y}  \\
    & =  \frac{(729 y^3 + 243 y^2 - 9 y + 1) \mp \sqrt{(729 y^3 - 81 y^2 + 27 y + 1)^2}}{6 y (9 y + 1) (9 y - 1)}
\end{aligned}
\end{align}
By choosing one of $\alpha_{\pm}$ appropriately, we obtain
\begin{align}
    \frac{(729 y^3 + 243 y^2 - 9 y + 1) + (729 y^3 - 81 y^2 + 27 y + 1)}{6 y (9 y + 1) (9 y - 1)} = \frac{81 y^2 + 1}{3 y (9 y - 1)} = 3 + \frac{9 y + 1}{3 y (9 y - 1)}.
\end{align}
Since $y$ is an odd integer, putting $y = 2 n + 1$, we conclude that every rational number of the form
$$\frac{162 n^2 + 162 n + 41}{3 (2 n + 1) (9 n + 4)} = 3 + \frac{9 n + 5}{3 (2 n + 1) (9 n + 4)}$$
is non-free for any $n \in \ZZ$. This completes the proof for the second family.
\end{proof}

\begin{rmk}
\normalfont
From the proof, we can recover an element $ g \in G_{\alpha}$ so that $g \cdot 0 = \frac{1}{2}$.
For the first family of non-free rational numbers in Theorem \ref{ExplicitSeqsOfRRNConvTo3}, we put $y = 2 n$ and 
$$x = -\frac{3}{2} (9 y - 2) = -3 (9 n - 1).$$
Therefore, for all $n \in \ZZ - \{ 0 \}$,
$$B_\alpha A^{-1} B_\alpha^{2 n} A^{-3 (9 n - 1)} \cdot 0 = \frac{1}{2} \qquad \left( \alpha = 3 + \frac{3}{2 (9 n - 1)} \right).$$
Notice that the equation also holds for $n = 0$.

For the second family of non-free rational numbers in Theorem \ref{ExplicitSeqsOfRRNConvTo3}, we put $y = 2 n + 1$ and 
$$x = -\frac{3}{4} (9 y - 1) (9 y + 1) = -3 (9 n + 4) (9 n + 5).$$
Therefore, for all $n \in \ZZ$,
$$B_\alpha A^{-1} B_\alpha^{2 n + 1} A^{-3 (9 n + 4) (9 n + 5)} \cdot 0 = \frac{1}{2} \qquad \left( \alpha = 3 + \frac{9 n + 5}{3 (2 n + 1) (9 n + 4)} \right).$$
\end{rmk}


\bibliography{bib}
\bibliographystyle{abbrv}

\end{document}